\documentclass{amsart}

\usepackage{amscd, amsmath, amssymb, amsfonts, mathtools, verbatim,hyperref,derivative,amsaddr}
\usepackage{hyperref}

\usepackage{times}
\usepackage[cmtip,  all]{xy}
\usepackage{graphicx}
\usepackage[active]{srcltx}
\usepackage{color}
\usepackage{textcomp}

\numberwithin{equation}{section}

\newtheorem{theorem}[equation]{Theorem}
\newtheorem{corollary}[equation]{Corollary}
\newtheorem{lemma}[equation]{Lemma}
\newtheorem{proposition}[equation]{Proposition}
\newtheorem{question}[equation]{Question}

\newtheorem*{ntheorem}{Theorem}

\theoremstyle{remark}
\newtheorem{remark}[equation]{\bf Remark}
\theoremstyle{remark}
\newtheorem{example}[equation]{\bf Example}

\theoremstyle{remark}
\newtheorem{definition}[equation]{\bf Definition}

\newcommand{\N}{\mathbb{N}}

\newcommand{\C}{\mathbb{C}}

\newcommand{\ML}{\mathrm{ML}}

\newcommand{\Ufd}{\mathrm{UFD}}
\newcommand{\Exp}{\mathrm{EXP}}

\newcommand{\Tr}{\mathrm{tr.deg}}

\newcommand{\Lnd}{\mathrm{LND}}

\newcommand{\Ker}{\mathrm{Ker}}
\newcommand{\Pl}{\mathrm{pl}}

\title{A note on $\mathbb{G}_a$-actions in positive characteristic}
\author{P M S Sai Krishna}
\address{\noindent Department of Mathematics,  Indian Institute of Technology Bombay,  Powai,  Mumbai 400076,  India\\
E-mail: saikrishna183@gmail.com}

\begin{document}

\date{}
\maketitle
{\let\thefootnote\relax

{\let\thefootnote\relax}

\begin{abstract}
    In \cite{Miya3}, Miyanishi proved that the ring of invariants of any $\mathbb{G}_a$ action on $\mathbb{A}^3$ is $\mathbb{A}^2$, when the field $k$ has zero characteristic. However, it is not known if this result holds when $k$ has positive characteristic. We provide a sufficient condition under which this result holds in positive characteristic. We also prove the following results related to the rigidity of the ring of invariants of an exponential map of a polynomial ring. 
    \begin{enumerate}
        \item Let $B=R^{[n]}$, where $R$ is a $k$-domain and $\delta \in \mathrm{EXP}_R(B)$ is a triangular exponential map. Then $B^{\delta}$ is non-rigid. In particular, for any field $k$ of zero characteristic the kernel of any triangular $R$-derivation of $R^{[n]}$ is non-rigid. 

        \item Let $k$ be a field of zero characteristic and $R$ be a $k$-domain. Then the kernel of any linear locally nilpotent $R$-derivation of $R^{[n]}$ is non-rigid. 
    \end{enumerate} 
    When $k$ is an algebraically closed of zero characteristic, the commuting derivations conjecture for $k^{[3]}$ has been proved in \cite{SM}, and 
    it is shown in \cite{EK} that the weak Abhyankar Sathaye conjecture is equivalent to the commuting derivations conjecture.  By introducing the notion of commuting exponential maps and formulating the commuting exponential maps conjecture, we show that the weak Abhyankar-Sathaye conjecture is equivalent to the commuting exponential maps conjecture for any field of arbitrary characteristic. In particular, we prove the commuting derivations conjecture$(CD(3))$ for any field of zero characteristic. 
\end{abstract}

\section{Introduction}
\emph{Throughout the paper, $k$ denotes a field, $B,R$ are domains such that $k \subset R \subset B$, $\Lnd $ denotes locally nilpotent derivations and $\Lnd (B)$ denotes the set of all locally nilpotent derivations of $B$. Similarly, $\Exp (B)$ denotes the set of all exponential maps of $B$, and $B^{\delta}$ denotes the ring of invariants of $\delta \in \Exp (B)$. The characteristic of $k$ is zero when working with $\Lnd$ and arbitrary when working with exponential maps. $B^{[n]}$ denotes the polynomial ring in $n$ variables over $B$.}

Exponential maps of $k$-domains coincide with locally nilpotent derivations (\ref{rmk1}) when the characteristic of $k$ is zero. They are also known as locally finite iterative higher derivations. In particular, exponential maps of finitely generated $k$-domains are in bijective correspondence with $\mathbb{G}_a$ actions on the corresponding affine variety. Exponential maps ($\Lnd$) are an important tool in Affine Algebraic Geometry as demonstrated in \cite{ML}, \cite{NG1} and \cite{NG2}. The mere existence of an exponential map ($\Lnd $) on an affine $k$-domain imposes some structure on it. For instance, when $k$ is algebraically closed, and $A$ is a $k$-domain with $\Tr _kA=1$, then A has a non-trivial exponential map ($\Lnd $) if and only if $A=k^{[1]}.$  Thus, it is of importance to study whether a $k$-domain has a non-trivial exponential map ($\Lnd $). This problem is called the rigidity problem, where a $k$-domain $B$ is called rigid if it has no non-trivial exponential map ($\Lnd$) (or a non-trivial $\mathbb{G}_a$ action in case of varieties). One of the fundamental objects in the study of exponential maps is the ring of invariants (the kernel in the case of $\Lnd$). Thus, it is a natural consequence to study whether the ring of invariants of an exponential map (the kernel of a $\Lnd$) is rigid.

In \cite{Miya3}, Miyanishi showed that the kernel of any non-zero locally nilpotent derivation of $k^{[3]}$ is $k^{[2]}$. However, it is not known if the ring of invariants of a non-trivial exponential map of $k^{[3]}$ is always $k^{[2]}$. The notion of rank of an exponential map of $k^{[n]}$ can be defined similarly (\cite[$7$ of $2.3$]{PMS}) to the notion of rank of a locally nilpotent derivation of $k^{[n]}$ as defined in  \cite[$3.2.1$]{Gene} when the characteristic of $k$ is zero. It is easy to see that the ring of invariants of any rank $1$ exponential map of $k^{[n]}$ is $k^{[n-1]}$. In \cite{KW}, it is shown that when $R$ is an $HCF$-ring, then the ring of invariants of any $\delta \in \Exp _R(R[X,Y])$ is $R^{[1]}$ and from which it follows that any rank $2$ exponential map of $k^{[n]}$ has ring of invariants $k^{[n-1]}\ $ \cite[Proposition $3.4$]{PMS}. In particular, the ring of invariants of any rank $2$ exponential map of $k^{[3]}$ is $k^{[2]}$. In \cite{Koj}, it is shown that if the ring of invariants of any $\delta \in \Exp (k^{[3]})$ is non-rigid, then it is $k^{[2]}$. This leads to the following question. 
\begin{question}
    Is the ring of invariants of a non-trivial exponential map of $k^{[3]}$ always non-rigid?
\end{question}
In section $3$, we answer this question partially by providing a sufficient condition under which the ring of invariants of an exponential map of a $k$-domain is non-rigid. The key idea is to use the local property of rigidity (\ref{local rigid}). As a consequence of a general result (\ref{main}), we get the following result. 

\begin{ntheorem}[Theorem $(\ref{plinthk3})$]
    Let $B=k^{[3]}$ and $\delta \in \Exp (B)$. If the plinth ideal ($pl(\delta)$) contains a quasi-basic element, then $B^{\delta}= k^{[2]}$. 
\end{ntheorem} 

We note that for locally nilpotent derivations of $k^{[n]}$, Freudenburg posed a similar question in \cite[Question $11.1$]{Gene}. It is stated as follows. 
\begin{question}\label{q2}
    Let $D\in \Lnd (k^{[n]})$ and $n\geq 4.$ Is $\ker (D)$ always non-rigid? 
\end{question}
In section $4$, by finding a commuting locally nilpotent derivation with a different kernel, we answer (\ref{q2}) affirmatively for all triangular derivations and linear locally nilpotent derivations of $k^{[n]}$, which follow as a consequence of a more general result  (\ref{rigid}).  These are the main results of section $4$.
\begin{ntheorem}[Theorem (\ref{tri})]
    Let $D\in \Lnd _R (R^{[n]})$ and $n\geq 2$.  If $D$ is triangular, then $\Ker (D)$ is non-rigid.
\end{ntheorem}

\begin{ntheorem}[Theorem (\ref{linear})]
    Let $n\geq 2$ and  $D\in \Lnd_R (R^{[n]})$ is linear. Then, $\Ker (D)$ is non-rigid. 
\end{ntheorem}

In section $5$, we introduce the notion of commuting exponential maps (\ref{commute}) and applying the ideas of section $4$ to exponential maps to prove the following result.
\begin{ntheorem}[Theorem (\ref{triangularexpo})]
     Let $B=R^{[n]}$, where $R$ is a $k$-domain, $n\geq 2$ and $\delta \in \Exp_R (B)$ be a triangular exponential map. Then $B^{\delta}$ is non-rigid. 
\end{ntheorem}

In section $6$, we expand on the notion of commuting exponential maps defined in section $5$. We recall the commuting derivations conjecture which is a well-known conjecture related to locally nilpotent derivations. 

\textbf{Commuting derivations conjecture CD(n):} Let $B=k[X_1,\dots,X_n]=k^{[n]}$ and $D_1,\dots,D_{n-1}\in \Lnd(B)$ be pairwise commuting and linearly independent (over $B$). Then $\bigcap_{i=1}^{n-1}\ker D_i=k[f]=k^{[1]}$, where $f$ is a coordinate of $B$.

In \cite{SM}, it was shown that the commuting derivations conjecture holds for $\C^{[3]}$. When $k$ is an algebraically closed field of zero characteristic, it was shown in \cite{EK} that the commuting derivations conjecture is equivalent to the weak Abhyankar-Sathaye conjecture, which is a weaker version of the Abhynakar-Sathaye conjecture. 

\textbf{Abhyankar-Sathaye Conjecture:} Let $f\in K[X_1,\dots,X_n]$ be such that $k[X_1,\dots,X_n]/(f)=k^{[n-1]}$. Then $f$ is a coordinate of $k[X_1,\dots,X_n]$.

\textbf{Weak Abhyankar-Sathaye Conjecture WAS(n):} Let $f\in K[X_1,\dots,X_n]$ be such that $k(f)[X_1,\dots,X_n]=k(f)^{[n-1]}$. Then $f$ is a coordinate of $k[X_1,\dots,X_n]$. 

We note that every $f$ satisfying the Abhyankar-Sathaye conjecture satisfies the weak Abhyankar-Sathaye conjecture.   Although there are counter-examples for the Abhyankar-Sathaye conjecture (\cite{Seg} and \cite{Nag}) when $k$ is of positive characteristic, however, it is true for some special cases.  We refer the reader to \cite{DG} for a survey of partial results on the  Abhyankar-Sathaye conjecture. The weak Abhyankar-Sathaye conjecture holds for $n=3$ and was proved in \cite{Kal}  when $k=\C$ and  in \cite{DK} when $k$ is any field of zero characteristic. However, it not known in the case when $k$ is of positive characteristic. 

In section $4$, we introduce the notion of pairwise commuting irredundant locally nilpotent derivations (\ref{irredLnd}) and show that this property is equivalent to the existing notion of pairwise commuting linearly independent locally nilpotent derivations (\ref{capture}). As a consequence of this (\ref{capture}), in the case of zero characteristic the equivalence of commuting derivations conjecture and weak Abhyankar-Sathaye conjecture proved in \cite{EK} holds even when we replace the pairwise commuting linearly independent locally nilpotent derivations with pairwise commuting irredundant locally nilpotent derivations. This motivates the definition of irredundant pairwise commuting exponential maps (\ref{irred}) and we formulate the commuting exponential maps conjecture as follows. 

\textbf{Commuting exponential maps Conjecture CE(n):} Let $B=k[X_1,\dots,X_n]=k^{[n]}$ and $\delta_1,\dots,\delta_{n-1} \in \Exp (B)$ be a set of irredundant commuting exponential maps (\ref{irred}). Then $\bigcap_{j=1}^{n-1}B^{\delta_j}=k[f]$, where $f$ is a coordinate of $B$.

 Following the ideas from \cite{EK}, we prove that the commuting exponential maps conjecture is equivalent to the weak Abhyankar-Sathaye conjecture for any field of arbitrary characteristic (\ref{equivalence}). As a consequence, we deduce that the commuting locally nilpotent derivations conjecture is true for $n=3$ for any field $k$ of zero characteristic (\ref{CD3}).

\section{Preliminaries}

In this section, we recall some definitions related to exponential maps. For further details related to them, we refer the reader to \cite{CM} and \cite[Chapter $1$]{Miya2}. We will use the following definition of an exponential map.

\begin{definition}\label{def}Let $k \subset R \subset B$.

An $R$-\emph{Exponential map} $(\delta)$ on $B$ is a set of $R$-linear endomorphisms $\{D_n\}_{n=0}^{\infty}$  of $B$ satisfying the following conditions. 
 \begin{enumerate}
        \item $D_0$ is the identity map on $B$.

        \item $\forall b\in B$, there exists an integer $n>0$ such that $D_i(b)=0$ for all $i\geq n$.

        \item For all $x,y \in B$ and for all integer $n\geq 0$, $D_n(xy)=\sum_{0\leq i\leq n}D_i(x)D_{n-i}(y)$.

        \item for all integers $i,j\geq 0$, $D_j \circ D_i=\binom{i+j}{j}D_{i+j}$.
 \end{enumerate}
\end{definition}

\begin{remark}
    \begin{enumerate}
    \item The $R$-linear endomorphisms $D_n$ are called \emph{higher derivatives} of $B$. We note that $D_1$ is a \emph{locally nilpotent derivation of $B$}.

    \item We can view $\delta$ as a $R$-algebra homomorphism from $B \longrightarrow B^{[1]}=B[t]$  and express $\delta$ as follows. $$\delta_t(x)=D_0(x)+D_1(x)t+D_2(x)t^2+\cdots+D_m(x)t^m $$
    
    \item  We denote the set of all $R$-exponential maps on $B$ by $\Exp _R(B)$. The set $B^{\delta}=\{x\in B|\  \delta(x)=x\}$ is called the \emph{ring of $\delta$-invariants} of $B$. We call $\delta \in \Exp _R(B)$ \emph{non-trivial} if $B^{\delta}\neq B$. 
    
    \item  We denote by $lc_{\delta}(x)$ the leading coefficient of $\delta(x)$ which is given by $D_m(x)$, where $m=\deg_t(\delta(x))$. 

    \item When $R=k$ is a field the set of all exponential maps of $B$ is denoted by $\Exp (B)$
   
\end{enumerate}
\end{remark}

\begin{example}\label{example}
    Let $B=k[X_1, \dots,  X_n]$ be the polynomial ring in $n$ variables. Let $\delta_i \in \Exp (B)$ be such that  $\delta_i(X_j)=X_j$ $\forall \ j\neq i$ and  $\delta_i(X_i)=X_i+t$. Such $\delta_i$ are called \emph{shift} exponential maps on the polynomial ring. 
\end{example}

\begin{definition}\label{basic}
Let $\delta \in \Exp (B)$. 
\begin{enumerate}
    \item  Any $x\in B$ such that $\deg_t\delta(x)$ has minimal positive $t$-degree is called a \emph{local slice} of $\delta$. Every non-trivial exponential map has a local slice.

     \item  An element $x\in B$  is called a \emph{slice} if $x$ is a local slice of $\delta$ and $lc_{\delta}(x)$ is a unit.
    \item The \emph{Makar-Limanov invariant} is a $k$-subalgebra of $B$ defined as $\ML (B)=\bigcap _{\delta \in \Exp (B)}B^{\delta}.$

    \item $B$ is called \emph{rigid} if there does not exists any non-trivial exponential map on $B$. This is equivalent to the condition $\ML (B)=B$.

    \item $B$ is $\emph{semi-rigid}$  if  every non-trivial exponential map of $B$ has the same ring of invariants. 

    \item A subring $A$ of $B$ is said to be \emph{factorially closed} in $B$ if for any non-zero $x, y \in B$ such that $xy\in A$,  then $x, y \in A$. Note that if $A$ is factorially closed in $B$, then $A$ is algebraically closed in $B$.

    \item The \emph{plinth ideal} associated to $\delta$ is defined as $$pl(\delta)=\{lc_{\delta}(x)| \textit{where $x$ is a local slice of } \delta \} \cup \{0\}.$$  It follows from $(1)$ of (\ref{first principles}) that the plinth ideal $(\Pl (\delta))$ is an ideal of $B^{\delta}$. This was introduced in \cite{YN} and referred to as the plinth ideal in \cite{SK}.

    \item Let $B=R^{[n]}$ and $\delta \in \Exp _R(B)$. We say $\delta=\{D_n\}_{n=0}^{\infty}$ is \emph{triangular over $R$} if there exists a coordinate system $\{X_1,\dots,X_n\}$ of $B$ over $R$ such that 
    \begin{enumerate}
        \item  $\delta(X_i)-X_i \in tR[X_1,\dots,X_{i-1}][t]$ for all $1\leq i \leq n$.

        \item If $i$ is the smallest integer such that $\delta(X_i)\neq X_i$, then $X_i$ is a local slice of $\delta$.
    \end{enumerate}
    The conditions $(a)$ is equivalent to $D_j(X_i)\in R[X_1,\cdots, X_{i-1}]$ for all $1\leq i\leq n$ and $j\geq 1$.

\end{enumerate}   
\end{definition}

The following definition was introduced \cite[$1.4$]{DD} for elements in $k^{[3]}$ when the characteristic of $k$ is zero. We define a modified version for elements in any $k$-domain without requiring the irreducible condition.
\begin{definition}
 A non-zero element $x \in B$ is called a \emph{quasi-basic} element if there exists two non-trivial  exponential maps $\delta ,\epsilon $ of $B$  such that $B^{\delta}\neq B^{\epsilon}$ and $x\in  B^{\delta}\cap B^{\epsilon}.$   
\end{definition}

We summarise below some useful properties of exponential maps \cite[Lemma $2.1,  2.2$]{CM} and \cite[$1.3.1$ and $1.5$]{Miya2}.
\begin{proposition}\label{first principles}
Let $\delta$ be a non-trivial exponential map on a $k$-domain $B$ and $x$ be a local slice with $n=\deg_{t}\delta(x)$. Let $c=lc_{\delta}(x)=D_n(x)$, where $D_i$'s are as in (\ref{def}).
\begin{enumerate}
    \item $D_i(x)\in B^{\delta}$ for all $i>0.$ 
    
    \item  $B^{\delta}$ is factorially closed in $B$. In particular,  $B^{\delta}$ is algebraically closed in $B$. 

    \item \label{Slice} $B[c^{-1}]=B^{\delta}[c^{-1}][x]$. Moreover, if $x$ is a slice then $B=B^{\delta}[x]$. This is called the \emph{Slice theorem.} 
    
    \item  $\Tr _{B^{\delta}}B=1$.

    \item Let $A=B^{\delta}$ and $S\subset A\setminus\{0\}$ be a multiplicative closed subset. Then $\delta$ extends to a non-trivial exponential map $S^{-1}\delta$ on $S^{-1}B$ defined as $S^{-1}\delta (\frac{b}{s})=\frac{\delta(b)}{s}$ for all $b\in B$ and $s\in S$. Moreover, $(S^{-1}B)^{S^{-1}\delta}=S^{-1}(B^{\delta})$.
\end{enumerate}
\end{proposition} 

We require the following two results related to exponential maps from \cite[Theorem $3.1$]{CM} and \cite[Theorem $3.1$]{Koj}, respectively.  
\begin{theorem}[\emph{Semi-Rigidity Theorem}]\label{semi}
Let $B$ be an affine $k$-domain. Then $B$ is rigid if and only if $B^{[1]}$ is semi-rigid.  Equivalently,  $B \textit{ is rigid}$ if and only if $\ML (B^{[1]})=B.$ 
\end{theorem}

\begin{theorem}\label{kojima}
    Let $B=k^{[3]}$ and a non-trivial $\delta \in \Exp (B)$. If $B^{\delta}$ is non-rigid or if the characteristic of $k$ is zero, then $B^{\delta}=k^{[2]}$.
\end{theorem}

We recall the following theorem \cite{OZ} due to Zariski.  
\begin{theorem}[Zariski's Finiteness Theorem]\label{zariski}
    Let $A$ be a normal affine domain over a field $k$, and $L$ be a field such that $k\subset L \subset frac(A)$. If $\Tr _k(L)\leq 2$, then $L\cap A$ is affine over $k$.
\end{theorem}

We require the following result related to exponential maps from  \cite[Corollary $3.3$]{SK}.

\begin{corollary}\label{gufd slice}
      Let $B$ be a $\Ufd$ over a field $k$ with $\Tr _kB=2$ and $\delta \in \Exp (B)$ is non-trivial. Suppose $B$ is geometrically factorial. Then $B=(B^{\delta})^{[1]}.$
\end{corollary}

We need the following result which is a direct consequence of \cite[Proposition $2.1$ and $2.2$]{JX}.

\begin{proposition}\label{gfac}
     Let $B$ be a $\Ufd$ over a field $k$ and $\delta \in \Exp(B)$. If $B$ is geometrically factorial, then $B^{\delta}$ is geometrically factorial. 
\end{proposition}

We recall the following result from \cite[Proposition $4.8$]{AEH}. 
\begin{proposition}\label{tr1}
    Let $A$ be an $\mathrm{HCF}$-ring and $B=A[X_1,\dots,X_n]=A^{[n]}$. Suppose $R$ is a ring and  $\Tr _AR=1$ such that $A\subset R\subset B.$ If $R$ is factorially closed in $B$, then $R= A^{[1]}$.
\end{proposition}

\section{Ring of invariants of an exponential map of a polynomial ring}
We require the following technical lemma. It is similar to \cite[Chapter $1$, Principle $7$]{Gene} when the characteristic of $k$ is zero.   
\begin{lemma}\label{tech}
        Let  $\delta \in \Exp (B)$ be non-trivial and given by 
        $$\delta(x)=x+D_1(x)t+D_2(x)t^2+\cdots +D_k(x)t^k.$$ Suppose $c\in B^{\delta}\setminus\{0\}$. Then
        \begin{enumerate}
            \item $c\delta$ defined by 
             $$c\delta(x)=x+cD_1(x)t+c^2D_2(x)t^2+\cdots +c^kD_k(x)t^k,$$
              is an exponential map of $B$.

            \item $B^{\delta}=B^{c\delta}$.
        \end{enumerate}
\end{lemma}

\begin{proof}
    \emph{$(1):$} Let $\{D'_n\}_{n=0}^{\infty}$ be higher derivatives associated to $c\delta$. Then, $D'_n=c^nD_n$ is a $k$-linear endomorphism of $B$, and $D_0'=D_0$ is the identity map. We note that $D_n'$ satisfies property $(2)$ of (\ref{def}) for all $n$. 
    Property $(3)$ of (\ref{def}) is satisfied as follows. $$\sum_{0\leq i\leq n}D_i'(x)D_{n-i}'(y)=\sum_{0\leq i\leq n}c^iD_i(x)c^{n-i}D_{n-i}(y)=c^n \sum_{0\leq i\leq n}D_i(x)D_{n-i}(y)=c^nD_n(xy)=D_n'(xy).$$ For all integers  $i,j\geq 0$, we have 
    $$D_i'\circ D_j'=c^iD_i\circ c^jD_j=c^{i+j}D_i\circ D_j=c^{i+j}\binom{i+j}{j}D_{i+j}=\binom{i+j}{j}D_{i+j}' ,$$ where the second equality uses the fact that $c\in B^{\delta}$. Since $\{D_n'\}_{n=0}^{\infty}$ satisfy $(1)$-$(4)$ of (\ref{def}), it follows that $c\delta$ as defined above is an exponential map of $B$.

    \emph{$(2):$} This follows since $B^{c\delta}=\cap_{n\geq 1}\Ker(D_n')$ and $D_n'=c^nD_n$.
\end{proof}

We require the following property, which is a consequence of the previous lemma. 
\begin{lemma}\label{restrict}
    Let $B$ be an affine $k$-domain and $S\subset B\setminus\{0\}$ be a multiplicatively closed subset. Let $\delta \in \Exp (S^{-1}B)$ . Then there exists $c\in S$ such that $c\delta \in \Exp(S^{-1}B)$ as defined in (\ref{tech}), restricts to $B$ i.e., $c\delta \in \Exp (B)$.
\end{lemma}

\begin{proof}
    Let $b_1,\dots,b_n$ be the generators of $B$. We have 
    $$\delta(b_i)=b_i+D_1(b_i)t+D_2(b_i)t^2+\cdots+D_{k_i}(b_i)t^{k_i},$$ where $D_l(b_i)\in S^{-1}B$ for all $l$. We note that there exists a $c\in S$ such that $c^l D_{l}(b_i)\in B$ for all $l$. Hence, $c\delta(b_i)\in B$ for all $i$, where $c\delta \in \Exp(S^{-1}B)$ as defined in (\ref{tech}). Since $c\delta$ is a $k$-algebra homomorphism, it follows that $c\delta(B) \subset B$ and hence $c\delta \in \Exp(B)$.  
\end{proof}

We require the following lemma, which is proved in \cite[$2.22$]{Gene} when the characteristic of $k$ is zero. We will use a similar argument to prove it when $k$ is of arbitrary characteristic. 
\begin{lemma}\label{local lemma}
    Let $S\subset \ML (B)\setminus \{0\}$ be a multiplicatively closed subset. Then 
    \begin{enumerate}
        \item $\ML (S^{-1}B)\subset S^{-1}\ML (B)$.

        \item If $B$ is an affine $k$-domain, then $\ML (S^{-1}B)=S^{-1}\ML (B)$.
    \end{enumerate}
\end{lemma}

\begin{proof}
\textit{(1)} Let $\frac{b}{s}\in \ML (S^{-1}B)$, where $b\in B$ and $s\in S$. We will show that $b\in \ML (B)$. Let $\delta \in \Exp (B)$. Then by $(5)$ of \ref{first principles}, $\delta$ extends to a non-trivial exponential map on $S^{-1}B$ and $S^{-1}\delta(\frac{b}{s})=\frac{\delta(b)}{s}=\frac{b}{s}$ and hence $\delta(b)=b$ and thus $b\in \ML (B)$.

\textit{(2)}  Let $\frac{b}{s} \in  S^{-1}\ML (B)$, where $b\in \ML (B)$ and $s\in S$. Since $s$ is a unit of $S^{-1}B$, it is enough show that $b \in \ML (S^{-1}B).$ Let $\delta \in \Exp (S^{-1}B)$. Since $B$ is affine, by (\ref{restrict}) there exists $c\in S$ such that $c\delta \in \Exp (B)$. Since $b\in \ML (B)$, we get $(c\delta) (b)=b$ and hence $\delta(b)=b$ and thus $b \in \ML(S^{-1}B)$. Combining with $(1)$, the result follows. 
\end{proof}
The following result is a consequence of the above lemma and is proved in \cite[$2.23$]{Gene} when the characteristic of $k$ is zero.
\begin{corollary}\label{local rigid}
    Let $B$ be an affine $k$-domain and $S= B\setminus \{0\}$ be any multiplicatively closed set. 
    \begin{enumerate}
        \item If $B$ is rigid, then $S^{-1}B$ is rigid.  

        \item If $B$ is semi-rigid, then $S^{-1}B$ is semi-rigid. 
    \end{enumerate}
\end{corollary}

\begin{proof}
    \emph{$(1):$} This follows directly from $(2)$ of (\ref{local lemma}). 

    \emph{$(2):$} We can assume $B$ is semi-rigid but not rigid.  We are done if $S^{-1}B$ is rigid. Suppose $S^{-1}B$ is not rigid. Let $\epsilon \in \Exp (S^{-1}B)$ be non-trivial and note that $S\subset (S^{-1}B)^{\epsilon}$.  Since $B$ is affine, by (\ref{restrict}), there exists  $c\in S$ such that $\delta=c\epsilon \in \Exp(B)$. Since $B$ is semi-rigid, we have $B^{\delta}=\ML(B)$ and since $S \subset B^{\delta}$ , hence $S\subset \ML(B)$. By $(2)$ of \ref{local lemma}, we have $\ML(S^{-1}B)=S^{-1}B^{\delta}$. By $(5)$ of (\ref{first principles}), we have that $S^{-1}\delta \in \Exp (S^{-1}B)$ and $(S^{-1}B)^{S^{-1}\delta}=S^{-1}B^{\delta}$.  Thus $\ML(S^{-1}B)=(S^{-1}B)^{S^{-1}\delta}$. By $(4)$ of (\ref{first principles}), it follows that $S^{-1}B$ is semi-rigid. 
\end{proof}

\begin{corollary}
    Let $B$ be an affine $k$-domain and  $S\subset \ML (B)\setminus \{0\}$ be a multiplicatively closed subset. If $B$ is non-rigid, then $S^{-1}B$ is non-rigid. 
\end{corollary}

\begin{proof}
    By $(2)$ of (\ref{local lemma}), we have $\ML(S^{-1}B)=S^{-1}\ML (B)$. Suppose $S^{-1}B$ is rigid, then $\ML (S^{-1}B)=S^{-1}B=S^{-1}\ML(B)$. Consider $\tfrac{b}{1}\in S^{-1}B$. Since $\tfrac{b}{1}\in S^{-1}\ML(B)$, we have $\tfrac{b}{1}=\tfrac{x}{s}$ where $x\in \ML(B)$ and $x\in S$ which gives $bs=x\in \ML(B)$ and since $\ML(B)$ is factorially closed, $b\in \ML(B)$. Then $\ML(B)=B$ but this contradicts $B$ is non-rigid. Hence $S^{-1}B$ is rigid. 
\end{proof}

The following theorem gives a sufficient condition under which the ring of invariants of any non-trivial exponential map of any $k$-domain is non-rigid.

\begin{theorem}\label{main}
     Let $B$ be a $k$-domain and $\delta \in \Exp (B)$ be such that $B^{\delta}$ is an affine $k$-algebra. If the plinth ideal $(\Pl (\delta))$ contains a quasi-basic element, then $B^{\delta}$ is non-rigid.
\end{theorem}

\begin{proof}
    Let $x$ be a local slice of $\delta$ such that $c=lc_{\delta}(x)$ is a quasi-basic element. Since $c$ is a quasi-basic element so there exists non-trivial $\epsilon_1,\epsilon_2 \in \Exp (B)$ such that $c\in A_1 \cap A_2$, where $A_1=B^{\epsilon_1}, A_2=B^{\epsilon_2}$ and $A_1\neq A_2$.  By $(5)$ of (\ref{first principles}), $\epsilon_1$ and $\epsilon_2$ extend to non-trivial exponential maps on $B[c^{-1}]$ with ring of invariants $A_1[c^{-1}]$ and $A_2[c^{-1}]$ respectively.  
    We claim that $A_1[c^{-1}]\neq A_2[c^{-1}]$. Suppose if $A_1[c^{-1}]= A_2[c^{-1}]$, then $\frac{a}{1} \in A_2[c^{-1}]$ for any $a\in A_1$ which implies that $c^na\in A_2$ for some integer $n\geq 0$ but since $A_2$ is factorially closed we get that $a\in A_2$ and hence $A_1 \subset A_2$. By an identical argument, $A_2 \subset A_1$ and hence $A_1=A_2$ is a contradiction. This proves our claim from which it follows that $B[c^{-1}]$ is not semi-rigid. By $(3)$ of (\ref{first principles}), $B[c^{-1}]=B^{\delta}[c^{-1}][x]$. Using Semi-Rigidity theorem (\ref{semi}), it follows that $B^{\delta}[c^{-1}]$ is not rigid.  Now, by $(1)$ of (\ref{local rigid}), we get that $B^{\delta}$ is not rigid. 
\end{proof}

\begin{corollary}\label{normal}
    Let $B$ be a normal affine $k$-domain with $\Tr_k(B)=3$ and $\delta \in \Exp (B)$. If the plinth ideal ($pl(\delta)$) contains a quasi-basic element, then $B^{\delta}$ is non-rigid.  
\end{corollary}
 \begin{proof}
    By $(4)$ of (\ref{first principles}), we have $\Tr _kB^{\delta}=2$. Let  $L=frac(B^{\delta})$. We have $L\cap B=B^{\delta}$ and by Zariski's Finiteness theorem (\ref{zariski}), $B^{\delta}$ is affine. By \ref{main}, $B^{\delta}$ is non-rigid .
 \end{proof}

\begin{theorem}\label{plinthk3}
    Let $B=k^{[3]}$ and $\delta \in \Exp (B)$. If the plinth ideal ($pl(\delta)$) contains a quasi-basic element, then $B^{\delta}= k^{[2]}$. 

\end{theorem}
\begin{proof}
    Follows by \ref{normal} and \ref{kojima}.
\end{proof}
\noindent
The following result is an application of (\ref{main}). In section $5$, we prove a stronger version (\ref{triangularexpo}) of this result without requiring the condition that $B^{\delta}$ is an affine $k$-algebra. 
\begin{theorem}\label{triexpo}
    Let $B=k^{[n]},n\geq 2$ and $\delta \in \Exp (B)$ be a triangular exponential map. If $B^{\delta}$ is an affine $k$-algebra, then $B^{\delta}$ is non-rigid. 
\end{theorem}

\begin{proof}
    Suppose $\delta$ is triangular with respect to a system of variables $\{X_1,\dots,X_n\}$. Let $i$ be the smallest such that $\delta(X_i)\neq X_i$. Then $X_i$ is a local slice and $\delta(X_i)-X_i\in tk[X_1,\dots,X_{i-1}][t]$. We have $lc_\delta (X_i)\in k[X_1,\dots,X_{i-1}]$. Note that $lc_{\delta}(X_i)\in B^{\delta_i}$, where $\delta_i$ is the shift exponential map defined in (\ref{example}) and $B^{\delta_i}=k^{[n-1]}$. Suppose $B^{\delta}=B^{\delta_i}$, then clearly $B^{\delta}$ is non-rigid. If not, then $lc_{\delta}(X_i)$ is a quasi-basic element of $pl(\delta)$ and by (\ref{main}), $B^{\delta}$ is non-rigid. 
\end{proof}

\begin{remark}
\begin{enumerate}
    \item  The above result shows the existence of quasi-basic elements in the plinth ideal for triangular exponential maps, thus ensuring that it is not a vacuous condition.

    \item  By (\ref{triexpo}) and (\ref{kojima}), we get that the ring of invariants of any triangular exponential map of $k^{[3]}$ is $k^{[2]}$. This can also be deduced by \cite[Lemma $3.2$,  Proposition $3.4$ and Corollary $3.5$]{PMS}.
\end{enumerate}
\end{remark}

\section{Rigidity of the kernel of a locally nilpotent derivation of a polynomial ring}
\emph{
Throughout this section, $k$ is of zero characteristic. If $ D\in \Lnd (B)$ such that $R\subset \Ker (D)$, then $D$ is called a locally nilpotent $R$-derivation and the set of locally nilpotent $R$-derivations is denoted by $\Lnd _R(B)$.} 

 A derivation $D$ is called a \emph{locally nilpotent derivation} if for all $b\in B$ there exists a positive integer $n$ such that $D^n(b)=0$. We recall some definitions related to locally nilpotent derivations. 
\begin{definition}Let $B$ be a $k$-domain.
    \begin{enumerate}

    \item $B$ is called \emph{rigid} if there does not exist any non-zero $D \in \Lnd (B)$.

    \item  Let $B=R^{[n]}$. A derivation $D \in \Lnd_R( B)$ is called \emph{triangular} if there exists a coordinate system $\{X_1,\dots,X_n\}$ over $R$ such that  $D(X_i)\in R[X_1,\dots,X_{i-1}]$ for $1\leq i\leq n$. 

    \item  Let $B=R^{[n]}$. A derivation $D \in \Lnd_R(B)$ is called \emph{linear} if there exists a coordinate system $\{X_1,\dots,X_n\}$ over $R$ such that $D(X_i)$ is a linear homogeneous polynomial in $X_1,\dots,X_n$ for all $1\leq i \leq n$.

    \item Let $D\in \Lnd(B)$. The \emph{plinth ideal} associated to $D$ is given by $DB\cap \Ker (D)$. It is denoted by $\Pl(D)$.

    \end{enumerate}

\end{definition}

\begin{remark}\label{rmk1}
    Let $B$ be a $k$-domain and $D\in \Lnd(B)$. We define a $k$-algebra homomorphism 
    $\delta :B\longrightarrow B^{[1]}=B[t]$  defined by 
    $$\delta (b)=\sum_{i=0}^{n}\tfrac{D^n(b)}{n!}t^n.$$ We note that $\delta \in \Exp (B)$  and $B^{\delta}=\ker (D)$. Moreover, if $\epsilon=\{F_n\}_{n=0}^{\infty} \in \Exp (B)$, then $F_1$ is a locally nilpotent derivations, $F_n=\tfrac{F_1^n}{n!}$, and $B^{\epsilon}=\ker (F_1)$. Thus, we have a bijective correspondence between $\Exp (B)$ and $\Lnd (B)$. 
\end{remark}

\noindent

The following two results are restatements of  (\ref{main}) and (\ref{normal}) in the language of $\Lnd$. 

\begin{corollary}
    Let $D \in \Lnd(B)$ such that $\Ker (D)$ is an affine $k$-algebra. If the plinth ideal ($\Pl(D)$) contains a quasi-basic element, then $\Ker (D)$ is non-rigid. 
\end{corollary}

\begin{corollary}
    Let $B$ be a normal affine $k$-domain with $\Tr _K (B)=3$ and $D \in \Lnd (B)$. If the plinth ideal ($\Pl(D)$) contains a quasi-basic element, then $\Ker (D)$ is non-rigid. 
\end{corollary}

\begin{proposition}\label{rigid}
Let  $D\in \Lnd (B)$ be non-zero. If there exists a non-zero $E \in \Lnd (B)$ such that $DE=ED$ and $\Ker (D) \neq \Ker (E) $. Then, $\Ker (D)$ is non-rigid. 
\end{proposition}
\begin{proof}
    Let $A=\Ker (D)$ and $A'=\Ker (E)$. Since $DE=ED$, $E(A)\subset A$ and hence $E$ is a locally nilpotent derivation of $A$. We claim that $E(A)\neq 0$. Suppose $E(A)=0$, then $A\subset A'$. By (\ref{first principles}) $A$ and $A'$ are factorially closed in $B$ with $\Tr _AB=\Tr _{A'}B=1$, it follows that  $A=A'$ which contradicts $A\neq A'$ Thus, $E\in \Lnd (A)$ is non-zero; hence, $A$ is non-rigid. 
\end{proof}

\begin{corollary}\label{coro}
    Let $n\geq 2$ and $ D=\sum_{i=1}^n f_i\pdv{}{X_i} \in \Lnd _R (R[X_1,\cdots,,X_n])$.
    \begin{enumerate}
        \item If for some $1\leq j\leq n, \pdv{f_i}{X_j}=0$ for all $1\leq i  \leq n$. Then, $\Ker (D)$ is non-rigid.  

        \item If $\mathrm{div}(f_i)=\sum_{j=1}^n\pdv{f_i}{X_j}=0$ for all $1\leq i \leq n$. Then, $\Ker (D)$ is non-rigid.
    \end{enumerate}

\end{corollary}

\begin{proof}
\emph{$(1):$} Let $E=\pdv{}{X_j}$ and $A=\Ker (D)$. We have that $DE=ED$ and $\Ker (E)=R^{[n-1]}$. If $A=\Ker (E)$, then clearly $A$ is non-rigid. If $A\neq \Ker (E)$, then by (\ref{rigid}), $A$ is non-rigid. 

\emph{$(2):$} Let $E=\sum_{i=1}^{n}\pdv{}{X_i}$ and $A=\Ker (D)$.We note that $DE=ED$. Since $E^2(X_i)=0$ for all $i$, $E \in \Lnd (R^{[n]})$. From the first part, $\Ker (E)$ is non-rigid. If $A=\Ker (E)$, then clearly $A$ is non-rigid. If $A\neq \Ker (E)$, then by (\ref{rigid}), $A$ is non-rigid. 
\end{proof}

\begin{theorem}\label{tri}
    Let $D\in \Lnd _R (R^{[n]})$ and $n\geq 2$.  If $D$ is triangular, then $\Ker (D)$ is non-rigid.
\end{theorem}

\begin{proof}
    Let $D$ be triangular with respect to a coordinate system $\{X_1,\dots, X_n\}$. Let $f_i=D(X_i)\in R[X_1,\dots, X_{i-1}]$ for all $i$. Then $\pdv{f_i}{X_n}=0$ for all $1\leq i\leq n$ and the result follows by $(1)$ of (\ref{coro}). 
\end{proof}

\begin{remark}
 In \cite{DF}, a triangular derivation of $k^{[5]}$ is constructed whose kernel is not an affine $k$-algebra. We conclude that the rigidity of the kernel of a triangular derivation of $k^{[n]}$ is independent of whether the kernel is an affine $k$-algebra.
\end{remark}
\noindent
The following proposition shows that the kernel of some locally nilpotent derivations of the kernel of a triangular locally nilpotent derivation (constructed in \cite{DF}) are non-rigid. 

\begin{proposition}
    Let $B=k^{[5]}=k[x,s,t,u,v]$ and $T=x^3\frac{\partial}{\partial s}+s\frac{\partial}{\partial t}+t\frac{\partial}{\partial u}+x^2\frac{\partial}{\partial v}$. Let $A=\ker (T)$ and $E=\frac{\partial}{\partial v},F=\frac{\partial}{\partial u} \in \Lnd(k^{[5]})$. Then $E,F$ restrict to $A$ and $\ker (E|_A), \ker(F|_A)$ are non-rigid. 
\end{proposition}

\begin{proof}
    We note that $E, F$ commute with $T$. By (\ref{rigid}), $E$ and $F$ restricts to a non-zero locally nilpotent derivation on $A$. Let $R=\ker (E|_A)$ and $S=\ker (F|_A) $. We have that $R=\ker (E) \cap A=k[x,s,t,u]\cap A$ and $S=\ker (F)\cap A=k[x,s,t,v]\cap A$. Notice that $(s-xv)\in S$ but $(s-xv) \notin R$ and hence $R\neq S$. Thus, $E|_A$ and $F|_A$ are two commuting locally nilpotent derivations on $A$ and $R\neq S$. By (\ref{rigid}), $R,S$ are non-rigid.
\end{proof}

\begin{theorem}\label{linear}
    Let $n\geq 2$ and  $D\in \Lnd_R (R^{[n]})$ is linear. Then, $\Ker (D)$ is non-rigid. 
\end{theorem}

\begin{proof}
     Let $R^{[n]}=R[X_1,\dots,X_n]$ and $D(X_i)=\sum_{j=1}^na_{ij}X_j$. Let $A=(a_{ij})_{n \times n}\in M_n(R)$. We have $DX=AX$, where $X=[X_1,\dots,X_n]^T$ and hence $D^mX=A^mX$

     Since $D\in \Lnd _R(R^{[n]})$, there exists $N\in \N$ such that $D^N(X_i)=0$ for all $1\leq i \leq n$. It follows that $A^N=0$, i.e., $A$ is a nilpotent matrix and, in particular, $\mathrm{det}(A)=0$. By considering $A \in M_n(K)$, where $K=\mathrm{frac}(R)$ and clearing denominators, we can choose a non-zero vector $\Lambda=[\lambda_1,\dots,\lambda_n]^T$ with $\lambda_i\in R $ such that $A\Lambda=0$.  
    
    Define $E=\sum_{i=1}^n\lambda_i\pdv{}{X_i}$. Since $E^2(X_i)=0$ for all $i$, $E\in \Lnd_R (R^{[n]})$.  We have that $DE(X_i)=0$ for all $i$ and $ED(X_i)=\sum_{j=1}^n\lambda_ja_{ij}=0$ and thus $DE=ED$. Since $E$ is a triangular derivation, by (\ref{tri}), $\Ker (E)$ is non-rigid. If $\Ker (D)=\Ker (E)$, then clearly $\Ker (D)$ is non-rigid. If $\Ker(D) \neq \Ker (E)$,  then by (\ref{rigid}), $\Ker (D)$ is non-rigid. 
\end{proof}

\begin{definition}\label{irredLnd}
     Let $B$ be a $k$-domain. A set $\{D_1,\dots, D_m\}\in \Lnd (B)$ is  called a set of  \emph{irredundant pairwise commuting} locally nilpotent derivations  of $B$ if they are pairwise commuting and $\bigcap_{j\neq i}\ker D_j \nsubseteq \ker D_i$ for all $1\leq i\leq m$. 
\end{definition}

We recall the following result \cite[Lemma $3.1$]{SM} which was stated for $\C$-domains but holds for $k$-domains, where $k$ is any field of characteristic zero. 
\begin{lemma}
    Let $B$ be a $k$-domain and $D_1,\dots,D_n\in \Lnd(B)$ be pairwise commuting and  linearly independent (over $B$). Let $A_i=\ker D_i$ for $1\leq i\leq n$. Then for each $1\leq i\leq n$,  $\{D_j|_{A_i}|1\leq j\leq n , j\neq i\}$ is a set of pairwise commuting linearly independent (over $A_i$) locally nilpotent derivations of $A_i$. 
\end{lemma}

We show that the notion of irredundant pairwise commuting locally nilpotent derivations is equivalent to the notion of pairwise commuting linearly independent locally nilpotent derivations. 

\begin{proposition}\label{capture}
     Let $B$ be a $k$-domain and  $S=\{D_1,\dots,D_n |D_i\in \Lnd(B)\}$. Then 
     $$ S \textit{ is pairwise commuting linearly independent}  \iff S \textit{ is  irredundant pairwise commuting }.$$
\end{proposition}

\begin{proof}
    \emph{$\Rightarrow$}:
    Suppose $S$ is pairwise commuting linearly independent. Let $A_i=\ker D_i$ and $B_i=\bigcap_{j\neq i}A_j$. By symmetry, it is enough to show that $D_n|_{B_n}$ is non-zero. By the previous lemma,$\{D_2,\dots,D_n\}\in \Lnd (A_1)$ is pairwise commuting and linearly independent (over $A_1$). Applying the lemma again, $\{D_3,\dots,D_n\}\in \Lnd (A_1\cap A_2)$ is pairwise commuting and linearly independent (over $A_1\cap A_2$) and so on we get that $\{D_n\in \Lnd (B_n)\}$ is linearly independent over $B_n$ and hence $D_n|_{B_n}$ is nonzero. It follows that $S$ is irredundant pairwise commuting.

    \emph{$\Leftarrow$}: 
    Suppose $S$ is irredundant and pairwise commuting. If $\{D_1,\dots,D_n\}$ is not linearly independent (over $B$). Then we have $\sum_{j=1}^nb_jD_j=0$, where not all $b_j$ are zero. Suppose $b_i\neq 0$. Let $x\in \bigcap_{j\neq i}\ker (D_j)$, then we get $b_iD_i(x)=0$ and hence $\bigcap_{j\neq i}\ker (D_j)\subset \ker (D_i)$. This contradicts that $S$ is irredundant pairwise commuting and hence $\{D_1,\dots,D_n\}$ is linearly independent (over $B$). 

\end{proof}

\section{Commuting exponential maps}
\emph{Throughout this section, $B$ denotes a domain and $k\subset R\subset B$ }
In this section, we extend the ideas of the previous section to exponential maps. We introduce the notion of commuting exponential maps as follows. 
\begin{definition}\label{commute}
    Let  $\delta =\{D_n\}_{n=0}^{\infty}$, $\epsilon=\{E_n\}_{n=0}^{\infty}\in \Exp_R (B)$ are said to be \emph{commuting} if $D_iE_j=E_jD_i$ for all $i,j\geq 1$. 
\end{definition}

\begin{proposition}\label{commuting expo}
     Let $\delta \in \Exp_R (B)$ be non-trivial and given by $\delta=\{D_n\}_{n=0}^{\infty}$. Suppose there exists a non-trivial $\epsilon \in \Exp_R (B)$ where $\epsilon=\{E_n\}_{n=0}^{\infty}$ such that $\delta ,\epsilon$ commute and $B^{\delta}\neq B^{\epsilon}$. Then $\delta,\epsilon$ restrict to non-trivial $R$-exponential maps on $B^{\epsilon}$ and $B^{\delta}$, respectively. In particular, $B^{\delta}$ and $B^{\epsilon}$ are non-rigid. 
\end{proposition}

\begin{proof}
    We show that $\epsilon $ restricts to an exponential map of $B^{\delta}=\bigcap _{n=1}^{\infty}\Ker (D_n)$.  Let $x\in B^{\delta}$. Since $D_iE_j=E_jD_i$, we get $D_i(E_j(x))=0$ for all $i,j \geq 1$ which implies $E_j(x)\in B^{\delta}$ for all $j \geq 1$ and $x\in B^{\delta}$. Suppose $\epsilon|_{B^{\delta}}$ is trivial. Then $B^{\delta} \subset B^{\epsilon}$, by $(4)$ of (\ref{first principles}), $\Tr _{B^{\delta}}B=\Tr _{B^{\epsilon}}B=1$. By $(2)$ of (\ref{first principles}), it follows that $B^{\delta}=B^{\epsilon}$ which is a contradiction. Thus, $\epsilon|_{B^{\delta}}$ is a non-trivial exponential map of $B^{\delta}$ and hence $B^{\delta}$ is non-rigid. Similarly, we can show that $\delta|_{B^{\epsilon}}$ is non-trivial and hence $B^{\epsilon}$ is non-rigid.
\end{proof}

\begin{theorem}\label{triangularexpo}
    Let $B=R^{[n]}$, where $R$ is a $k$-domain, $n\geq 2$ and $\delta \in \Exp_R (B)$ be a triangular exponential map (defined in $(8)$ of (\ref{basic})). Then $B^{\delta}$ is non-rigid. 
\end{theorem}
\begin{proof}
    Let $\delta $ be triangular with respect to the coordinate system $\{X_1,\dots,X_n\}$ and $\delta =\{D_n \}$ be the $R$-linear endomorphisms associated to $\delta$. Since $\delta$ is triangular we have that, $D_j(X_k)\in R[X_1,\dots, X_{k-1}]$ for all $1\leq k \leq n$ and $j\geq 1$. Consider the shift exponential map  $\epsilon=\{E_n\}_{n=0}^{\infty} \in \Exp_R(B)$ which is defined by $\epsilon (X_i)=X_i$ for all $i<n$ and $\epsilon(X_n)=X_n+t$.  Then we have $E_j(X_k)=0$ for all $k<n, j\geq 1$ and $E_1(X_n)=1$ and $E_j(X_n)=0$ for all $j\geq 2$. We claim that $D_iE_j=E_jD_i$ for all $i,j\geq 1$. 

    We note that $D_iE_j-E_jD_i$ is an $R$-linear map of $B$ and  its kernel is an $R$-submodule of $B$, hence it is enough to show that $D_i(E_j(X_1^{e_1}\cdots X_n^{e_n}))=E_j(D_i(X_1^{e_1}\cdots X_n^{e_n}))$ for  $i,j\geq 1 ,$ and $ e_k\geq 0$ for all $1\leq k\leq n$ . We first show that $D_i(E_j(X_n^{e_n}))=E_j(D_i(X_n^{e_n}))$ for all $i,j\geq 0$ and $e_n\geq 1$. We will use induction on $i,j,e_n$.  The base case $i=j=0$ and $e_n=1$ is clearly true.

    Suppose this holds for all tuples $(i,j,e_n)$ where $i<a$ or $j<b$ or $e_n<m$. Then by the product and linearity properties of $D_a,E_b$ we have
    \begin{align*}
        D_a(E_b(X_n^m))-E_b(D_a(X_n^m))=D_a(\sum_{l=0}^bE_l(X_n^{m-1})E_{b-l}(X_n))-E_b(\sum_{k=0}^aD_k(X_n^{m-1})D_{a-k}(X_n))&\\
        =\sum_{l=0}^b D_a(E_l(X_n^{m-1})E_{b-l}(X_n))-\sum_{k=0}^aE_b(D_k(X_n^{m-1})D_{a-k}(X_n))&\\
        =\sum_{l=0}^{b}\sum_{k=0}^a D_k(E_l(X_n^{m-1}))D_{a-k}(E_{b-l}(X_n))-\sum_{k=0}^a\sum_{l=0}^bE_l(D_k(X_n^{m-1}))E_{b-l}(D_{a-k}(X_n))=0&\\
    \end{align*} 
    where the last equality follows by the induction hypothesis.  This proves that $D_i(E_j(X_n^{e_n}))=E_j(D_i(X_n^{e_n}))$ for all $i,j,e_n\geq 0$. Using this, we show that $D_i(E_j(X_1^{e_1}\cdots X_n^{e_n}))=E_j(D_i(X_1^{e_1}\cdots X_n^{e_n}))$ for  $i,j\geq 1 ,$ and $ e_k\geq 0$ for all $1\leq k\leq n$. 
    
    Since $X_1,\dots, X_{n-1} \in B^{\epsilon}$, we get that $R[X_1,\dots,X_{n-1}] \subset \Ker (E_j)$ for all $j \geq 1$.
    By using the product and linearity rules we have 
    \begin{align*}
        D_i(E_j(X_1^{e_1}\cdots X_n^{e_n}))&-E_j(D_i(X_1^{e_1}\cdots X_n^{e_n}))\\
        &=D_i(X_1^{e_1}\cdots X_{n-1}^{e_{n-1}}E_j(X_n^{e_n}))-E_j(\sum_{k=0}^iD_k(X_1^{e_1}\cdots X_{n-1}^{e_{n-1}})D_{i-k}(X_n^{e_n}))\\
        &=\sum_{k=0}^i D_k(X_1^{e_1}\cdots X_{n-1}^{e_{n-1}})D_{i-k}(E_j(X_n^{e_n}))-\sum_{k=0}^iE_j(D_k(X_1^{e_1}\cdots X_{n-1}^{e_{n-1}})D_{i-k}(X_n^{e_n}))\\
        &=\sum_{k=0}^iD_k(X_1^{e_1}\cdots X_{n-1}^{e_{n-1}})D_{i-k}(E_j(X_n^{e_n}))-\sum_{k=0}^i D_k(X_1^{e_1}\cdots X_{n-1}^{e_{n-1}})E_j(D_{i-k}(X_n^{e_n}))\\    
        &=0
    \end{align*} 
    where the third equality follows by noticing that $D_j(X_1^{e_1}\cdots X_{n-1}^{e_{n-1}}) \in R[X_1,\dots, X_{n-1}]$ and the last equality follows by using that $D_i(E_j(X_n^{e_n}))=E_j(D_i(X_n^{e_n}))$ for all $i,j,e_n\geq 0$. We conclude that $\delta, \epsilon$ are commuting exponential maps. Since $B^{\epsilon}=R^{[n-1]}$, if $B^{\delta}=B^{\epsilon}$, then clearly $B^{\delta}$ is non-rigid. If not, by (\ref{commuting expo}), $B^{\delta}$ is non-rigid.  
\end{proof}

\begin{remark}
The above proof doesn't use the condition $(b)$ of the definition ($8$ of \ref{basic}) of a triangular exponential map.     
\end{remark}

\section{Commuting exponential maps conjecture}
\emph{Throughout this section, $B$ denotes a $k$-domain}
\begin{definition}\label{irred}
   A set of  non-trivial exponential maps $\{\delta_1,\dots, \delta_m\}$ of $B$ is  called a set of  \emph{irredundant pairwise commuting} exponential maps of $B$ if they are pairwise commuting and $\bigcap_{j\neq i}B^{\delta_j}\nsubseteq  B^{\delta_i}$ for all $1\leq i\leq m$. 
\end{definition}

\noindent 
We make the following observation.
\begin{remark}\label{rmk2}
    Let $k$ be a field of characteristic zero and $B$ be a $k$-domain. Let $\{\delta_1,\dots,\delta_n\}$ be a set of irredundant pairwise commuting exponential maps of $B$. Let $D_i$ be the corresponding locally nilpotent derivations to $\delta_i$ as stated in (\ref{rmk1}). Then $\{D_1,\dots,D_n\}$ is a set of pairwise commuting linearly independent locally nilpotent derivations of $B$.  In particular, we have $CE(n)\iff CD(n)$.
\end{remark}

\noindent
The goal of this section is to prove the equivalence of the commuting exponential maps conjecture and the weak Abhyankar-Sathaye conjecture in arbitrary characteristic. We require the following observation about shift exponential maps. 

\begin{lemma}\label{shiftCommuting}
    Let $B=k^{[n]}=k[X_1,\dots,X_n]$ and shift exponential maps be as defined in (\ref{example}). Then any two shift exponential maps relative to the same coordinate system are commuting.  
\end{lemma}

\begin{proof}
      Let $\delta=\{D_a\}_{a=0}^{\infty}, \epsilon=\{E_b\}_{b=0}^{\infty}$ be shift exponential maps as defined by $\delta(X_l)=X_l \ \forall l\neq i, \delta (X_i)=X_i+t$ and $\epsilon(X_l)=X_l \ \forall l\neq j, \epsilon (X_j)=X_j+t$, respectively. We show that $\delta, \epsilon$ are commuting exponential maps.  
      
     We note that $D_aE_b-E_bD_a$ is a $k$-linear map of $B$ and its kernel is a $k$-submodule of $B$ and hence it is enough to show that $D_a(E_b(X_1^{e_1}\cdots X_n^{e_n}))=E_b(D_a(X_1^{e_1}\cdots X_n^{e_n}))$ for  $a,b\geq 0 ,$ and $ e_l\geq 0$ for all $1\leq l\leq n$.  We note that $D_a$ is $B^{\delta}$-linear for all $a\geq 0$ and $E_b$ is $B^{\epsilon}$-linear for all $b\geq 0$. Since $X_l\in B^{\delta}\cap B^{\epsilon}$ for $l\neq i,j$, we have $$(D_aE_b-E_bD_a)(f)=0\ \forall f\in k[X_1,\dots,\hat{X_i},\dots,\hat{X_j},\dots,X_n].$$    
     Thus, it is enough to show that $(D_aE_b-E_bD_a)(X_i^{r} X_j^{s})=0$ for all $r\geq 0, s\geq 0$. We show this by induction on the tuples $(a,b,r,s)$. The base case $a=0,b=0,r=0,s=0$ is trivially true.  Suppose this holds for all tuples $(a',b',r',s')$, where $a'<a$ or $b'<b$ or $r'<r$ or $s'<s$.  We have
      \begin{align*}
        D_aE_b(X_i^r X_j^s)&=D_a(\sum_{l=0}^{b}E_l(X_i^rX_j^{s-1})E_{b-l}(X_j))\\
        &=\sum_{l=0}^bD_a(E_l(X_i^rX_j^{s-1}) E_{b-l}(X_j))\\
        &=\sum_{l=0}^bE_{b-l}(X_j)D_a(E_l(X_i^rX_j^{s-1}))\\
        &=\sum_{l=0}^bE_{b-l}(X_j)E_l(D_a(X_i^rX_j^{s-1}))\\
        &=E_b(X_jD_a(X_i^rX_j^{s-1}))\\
        &=E_b(D_a(X_i^rX_j^s))
    \end{align*} 
    Where the third equality follows by noticing that $E_{b-l}(X_j) \in B^{\delta}$ and the fourth equality follows by the induction hypothesis. Thus, we have $D_aE_b=E_bD_a$ for all $a, b\geq 0$, and hence $\delta,\epsilon $ are commuting exponential maps. 
\end{proof}

\begin{lemma}\label{commuting slices}
    Let $B$ be a $k$-domain and $\delta_1,\dots,\delta_m\in \Exp(B)$ be pairwise commuting. Suppose there exists $s_1,\dots, s_m \in B$ such that $s_i$ is a slice of $\delta_i$ and $s_i\in \bigcap_{j\neq i}B^{\delta_j}$. Then $B=A[s_1,\dots,s_m]$ where $A=\bigcap_{j=1}^mB^{\delta_j}$.
\end{lemma}
\begin{proof}
    We can prove this by induction. If $m=1$, then the result follows by Slice theorem. Suppose the result holds for $m-1$. Then $B=C[s_1,\dots,s_{m-1}]$, where $C=\bigcap_{j=1}^{m-1}B^{\delta_j}$. Since $\delta_i$ are pairwise commuting, by (\ref{commuting expo}), we get $\delta_m \in \Exp(B^{\delta_j})$ for all $j\neq m$ and  it follows that $\delta_m|_{C} \in \Exp(C)$. Since $ s_m\in  \bigcap _{j\neq m}B^{\delta_j}$ and $\delta_m(s_m)=1$, it follows that $\delta_m$ is non-trivial with a slice $s_m\in C$ and has ring of invariants $B^{\delta_m}\cap C=A$. By Slice theorem, $C=A[s_m]$ and hence $B=A[s_1,\dots,s_m]$.
\end{proof}

\begin{lemma}\label{trdrop}
    Let $\Tr_k B=n$ and $1\leq l\leq n$. Suppose $S=\{\delta_1,\dots,\delta_l\}$ is a set of irredundant pairwise commuting exponential maps of $B$ and $A=\bigcap_{j=1}^lB^{\delta_j}$. Then $\Tr _kA=n-l$.
\end{lemma}

\begin{proof}
    We prove using induction on $(n,l)$. The base case $(n=l=1)$ follows by $(4)$ of \ref{first principles}. Suppose the result hold for all pairs  $(a,b)$ with $a<n$ or $b<l$. Let $C=\bigcap_{j=1}^{l-1}B^{\delta_j}$.  By inductive hypothesis $\Tr_k C=n-l+1$. Since $S$ is a set of irredundant commuting exponential maps, hence $\delta_l|_{C}$ is a non-trivial exponential map with ring of invariants  $B^{\delta_l}\cap C=A$. By $(4)$ of (\ref{first principles}), it follows that $\Tr_k (A)=n-l+1-1=n-l$.
\end{proof}

\begin{lemma}\label{isk1}
     Let $B=k[X_1,\dots,X_n]=k^{[n]}$ and $\delta_1,\dots,\delta_{n-1} \in \Exp (B)$ be a set of irredundant pairwise commuting exponential maps. Then $\bigcap_{j=1}^{n-1}B^{\delta_j}=k[f]=k^{[1]}$, where $f\in \bigcap_{j=1}^{n-1}B^{\delta_j}$.
\end{lemma}

\begin{proof}
    Let $A=\bigcap_{j=1}^{n-1}B^{\delta_j}$.  By $(\ref{trdrop})$, $\Tr_k A=1$. Since each $B^{\delta_j}$ is factorially closed subring of $B$, it follows that $A$ is a factorially closed subring of $B$. By (\ref{tr1}), it follows that $A=k[f]=k^{[1]}$ for some $f\in A$. 
\end{proof}

\begin{remark}
    By the above lemma, to prove $CE(n)$ it remains to show that $f$ is a coordinate of $B$.
\end{remark}

\begin{theorem}\label{forward}
    $CE(n) \implies WAS(n)$.
\end{theorem}

\begin{proof}
    Let $B=k[X_1,\dots,X_n]$ and $f\in B$ such that $k(f)[X_1,\dots,X_n]=k(f)^{[n-1]}=k(f)[g_1,\dots,g_{n-1}]$. Let $A=k(f)[g_1,\dots,g_{n-1}]$. For $1\leq j \leq n-1$, let  $\delta_j$  be the shift exponential map of $A=k(f)[g_1,\dots,g_{n-1}]$ defined by $\delta_j(g_i)=g_i$ for all $i\neq j$ and $\delta_j(g_j)=g_j+t$. 
    
    \textbf{Claim 1:} There exists $h(f)\in k[f]\setminus \{0\}$ such that $h(f)\delta_j$ as defined in (\ref{tech}), restricts to a non-trivial exponential map of $B$ for all $1\leq j \leq n-1$. 

    \textbf{Proof of claim 1:}  Let $\delta_j=\{D_j^k\}_{k=0}^{\infty}$, where $D_j^k$ are the higher derivatives associated to $\delta_j$ . We have
    $$\delta_j(X_i)=X_i+D_j^1(X_i)t+\cdots+D_j^m(X_j)t^m ,$$
    where $D_j^k(X_j)\in A=k(f)[X_1,\dots,X_n]$ and any element in $k(f)[X_1,\dots,X_n]$ can be multiplied by a polynomial in $f$ so that it belongs to $B$. By the locally finite property of $\delta_j$, there exists some $h_j(f)\in k[f]\setminus\{0\}$ such that $h_j(f)\delta_j$ restricts to $B$. By taking $h(f)=h_1(f)\cdots h_{n-1}(f)$, we get that $h(f)\delta_j \in \Exp(B)$ for all $1\leq j \leq n-1$. 

    \textbf{Claim 2:} \{$h(f)\delta_j|1\leq j \leq n-1 \}$ form a set of irredundant pairwise commuting exponential maps of $B$.  
    
    \textbf{Proof of claim 2:} By (\ref{shiftCommuting}),  \{$h(f)\delta_j|1\leq j \leq n-1 \}$ are pairwise commuting exponential maps of $A$ and hence they are pairwise commuting exponential maps of $B$. Note that $A^{h(f)\delta_j}=A^{\delta_j}=k(f)[g_1,\dots,\hat{g_j},\dots,g_{n-1}]$ and $\bigcap_{j\neq i}A^{h(f)\delta_j}=k(f)[g_i]$. Since $B^{h(f)\delta_j}=A^{h(f)\delta_j}\cap B$, we get $\bigcap_{j\neq i}B^{h(f)\delta_j}=k(f)[g_i]\cap B$. 

    Suppose $\bigcap_{j\neq i}B^{h(f)\delta_j}=k(f)[g_i]\cap B \subset B^{h(f)\delta_i}=A^{h(f)\delta_i}\cap B$.  We know that there exists a polynomial in $f$ say $p(f)$ such that $p(f)g_i\in B$ and hence $p(f)g_i\in A^{h(f)\delta_i}\cap B$ and in particular, $p(f)g_i\in A^{h(f)\delta_i}$. Since $A^{h(f)\delta_i}$ is factorially closed, it follows that $g_i\in A^{h(f)\delta_i}=A^{\delta_i}$, which is a contradiction. This proves the claim.

    By claims $(1)$ and $(2)$,  \{$h(f)\delta_j|1\leq j \leq n-1 \}$ is a set of  irredundant commuting exponential maps on $B$. We have 
    $$\bigcap_{j=1}^{n-1}B^{h(f)\delta_j}=\bigcap_{j=1}^{n-1}A^{h(f)\delta_j}\cap B=(\bigcap_{j=1}^{n-1} A^{\delta_j})\cap B=k(f)\cap B.$$  Note that $k(f)\cap B=k[f]$ and hence $\bigcap_{j=1}^{n-1}B^{h(f)\delta_j}=k[f]$. It follows that if $CE(n)$ is true, then $f$ is a coordinate of $B$, and hence $WAS(n)$ is true.   
\end{proof}

\begin{proposition}\label{technical}
    Let $B$ be a $\Ufd$ over $k$ with $\Tr_k B=n$ and $B$ is geometrically factorial.  Let $\{\delta_1,\dots, \delta_{n-1}\}$ be a set of irredundant pairwise commuting exponential maps. Denote  $A=\bigcap_{j=1}^{n-1}B^{\delta_j}$ and  $A_i=\bigcap_{j\neq i}B^{\delta_j}$. Suppose $A=k[f]=k^{[1]}$ for some $f \in A$. Then 
    \begin{enumerate}

        \item For any $i=1,\dots,n-1$, there exists $g_i\in A_i$ and $h_i\in A=k[f]$ such that $A_i=k[f,g_i]=k^{[2]}$ and $lc_{\delta_i}(g_i)=h_i(f)$.

        \item  $B_{h(f)}=k[f]_{h(f)}[g_1,\dots,g_{n-1}]$, where $h(f)$ denotes the least common multiple of $h_i$'s from $(1)$.  In particular, we have $B_{k[f]\setminus\{0\}}=k(f)[s_1,\dots,s_{n-1}]$.
    \end{enumerate}
\end{proposition}

\begin{proof}
    
    \emph{$(1):$} By (\ref{trdrop}), $\Tr_k A_i=2$. 
    
    \textbf{Claim 1:} $A_i=(\bigcap_{j\neq i}B^{\delta_j})$ is geometrically factorial. 

    \textbf{Proof of claim 1:} Let $k'$ be any algebraic extension of $k$. By \cite[Proposition $2.1$]{JX}, $\delta_j\otimes 1 \in \Exp (B \otimes_k k')$ and by \cite[Proposition $2.3$]{JX}, $(B\otimes_k k')^{\delta_j \otimes 1}=B^{\delta_j}\otimes_k k'$. By $(2)$ of (\ref{first principles}), $B^{\delta_i}\otimes_kk'$ is factorially closed in $B\otimes_k k'$.  Using that $(\bigcap_{j\neq i}B^{\delta_j}) \otimes_k k'=\bigcap_{j\neq i}(B^{\delta_j}\otimes_k k')$, we conclude that $(\bigcap_{j\neq i}B^{\delta_j}) \otimes_k k'$ is factorially closed in $B\otimes_kk'$ and hence $(\bigcap_{j\neq i}B^{\delta_j}) \otimes_k k'$ is a $\Ufd$.  This proves the claim.

    Since $\{\delta_j\}$ are irredundant commuting exponential maps,  $\delta_i|_{A_i}$ is a non-trivial exponential map of $A_i$. By (\ref{gufd slice}), $A_i=(A_i^{\delta_i})^{[1]}$, where $A_i^{\delta_i}=B^{\delta_i}\cap A_i=A=k[f]$. We have $A_i=k[f,g_i]$, where $g_i \in A_i$. By \cite[Lemma $3.1$]{PMS}, $lc_{\delta_i}(g_i)=h_i(f)$, where $h_i\in A=k[f]$.

    \emph{$(2):$} Note that $h(f) \in A\subset B^{\delta_j}$ for $1\leq j\leq n-1$. By $(5)$ of (\ref{first principles}), $\delta_j$ extends to $\tilde{\delta_j} \in \Exp (B_{h(f)})$ with ring of invariants $(B_{h(f)})^{\tilde{\delta_j}}=(B^{\delta_j})_{h(f)}$.  We have $$\bigcap_{j=1}^{n-1}(B_{h(f)})^{\tilde{\delta_j}}=\bigcap_{j=1}^{n-1}(B^{\delta_j})_{h(f)}=(\bigcap_{j=1}^{n-1}B^{\delta_j})_{h(f)}=A_{h(f)}=k[f]_{h(f)}.$$ Note that  for $1\leq j \leq n-1, \tilde{\delta_{j}}$ are pairwise commuting exponential maps of $B_{h(f)}$ and $g_j$ is a slice of $\tilde{\delta_j}$.  By (\ref{commuting slices}), $B_{h(f)}=C[g_1,\dots,g_{n-1}]$, where $C=\bigcap_{j=1}^{n-1}(B_{h(f)})^{\tilde{\delta_j}}=k[f]_{h(f)}$ and hence the result follows. 
\end{proof}

\begin{corollary}\label{backward}
    $WAS(n) \implies CE(n)$.
\end{corollary}
\begin{proof}
    Let $B=k[X_1,\dots,X_n]$ and $\{\delta_1,\dots,\delta_{n-1}\}$ be a set of irredundant commuting exponential maps and let $A=\bigcap_{j-1}^{n-1}B^{\delta_j}$  By (\ref{isk1}), $A=k[f]=k^{[1]}$.  By $(2)$ of (\ref{technical}), we get $B_{k[f]\setminus \{0\}}=k(f)[X_1,\dots,X_n]=k(f)^{[n-1]}$. If $WAS(n)$ is true, then $f$ is a coordinate of $B$, and hence $CE(n)$ is true. 
\end{proof}

By (\ref{forward}) and (\ref{backward}), we have the following result. 
\begin{theorem}\label{equivalence}
    $CE(n) \iff WAS(n)$.
\end{theorem}

\begin{theorem}\label{CD3}
    The commuting derivations conjecture $CD(3)$ is true for any field of zero characteristic. 
\end{theorem}

\begin{proof}
    This follows by (\ref{equivalence}),\cite[Corollary $4.12$]{DK} and (\ref{rmk2}).
\end{proof}

{\bf Acknowledgement:} The author thanks Manoj K. Keshari for reviewing the first draft of this work and suggesting that (\ref{linear}) can be extended to $R^{[n]}$. A part of this work was done when the author was supported by the Prime Minister's Research Fellowship (ID:1301165).


\bibliographystyle{alpha}
\bibliography{refs} 

\end{document}